\documentclass{article}
\usepackage{amsmath,amsfonts,amsthm,amssymb,amscd,color,xcolor}
\usepackage{graphicx,amscd,mathrsfs,eufrak,wrapfig,mathrsfs,lipsum}
\usepackage{float}
\usepackage{tikz}
\usepackage{multicol}
\usepackage{caption}
\usetikzlibrary{arrows}
\usepackage{capt-of}
\usepackage{hyperref}

\colorlet{darkblue}{blue!50!black}

\hypersetup{
    colorlinks,%
    citecolor=darkblue,%
    filecolor=red,%
    linkcolor=darkblue,%
    urlcolor=magenta,%
    pdfnewwindow=true,%
    pdfstartview={FitH}
}

\newcommand{\p}{\partial}
\newcommand{\e}{\varepsilon}

\newcommand{\R}{{\mathbb R}}
\newcommand{\IP}{{\mathbb P}}

\newcommand{\SSS}{\boldsymbol{\mathit S}}

\newcommand{\XXX}{\boldsymbol{\mathit X}}

\newcommand{\aalpha}{{\boldsymbol\alpha}}
\newcommand{\mmu}{{\boldsymbol\mu}}

\newcommand{\eeta}{{\boldsymbol\eta}}
\newcommand{\zzeta}{{\boldsymbol\zeta}}

\newcommand{\BB}{{\cal B}}

\newcommand{\DD}{{\cal D}}

\newcommand{\KK}{{\cal K}}

\newcommand{\PP}{{\cal P}}
\newcommand{\RR}{{\cal R}}

\newcommand{\VV}{{\cal V}}
\newcommand{\WW}{{\cal W}}

\newcommand{\dd}{{\textup d}}

\newcommand{\PPPP}{{\mathfrak P}}

\newcommand{\uuu}{{\boldsymbol{\mathit u}}}

\newcommand{\supp}{\mathop{\rm supp}\nolimits}

\newcommand{\vol}{\mathop{\rm vol}\nolimits}

\newcommand{\Lie}{\mathop{\rm Lie}}
\newcommand{\lspan}{\mathop{\rm span}}

\theoremstyle{plain}

\newtheorem{theorem}{Theorem}[section]

\newtheorem{proposition}[theorem]{Proposition}

\theoremstyle{definition}

\theoremstyle{remark}

\newtheorem*{example*}{Example}

\numberwithin{equation}{section}

\begin{document}
\author{Armen Shirikyan} 
\title{Controllability implies mixing I.\\
Convergence in the total variation metric}
\date{\small D\'epartement de Math\'ematiques, Universit\'e de Cergy--Pontoise, 
CNRS UMR8088\\
2 avenue Adolphe Chauvin, 95302 Cergy--Pontoise Cedex, France\\ 
E-mail: \href{mailto:Armen.Shirikyan@u-cergy.fr}{Armen.Shirikyan@u-cergy.fr}\\
National Research University {\it Moscow Power Engineering Institute, Russia\/}}
\maketitle

\hfill{\sl In memory of my Teacher, Mark Iosifovich Vishik}

\medskip
\begin{abstract}
This paper is the first part of a project devoted to studying the interconnection between controllability properties of a dynamical system and the large-time asymptotics of trajectories for the associated stochastic system. It is proved that the approximate controllability to a given point and the solid controllability from the same point imply the uniqueness of a stationary measure and exponential mixing in the total variation metric. This result is then applied to random differential equations on a compact Riemannian manifold. In the second part, we shall replace the solid controllability by a stabilisability condition and prove that it is still sufficient for the uniqueness of a stationary distribution, whereas the convergence to it holds in the weaker dual-Lipschitz metric. 


\end{abstract}
\tableofcontents
\setcounter{section}{-1}

\section{Introduction}
\label{s0}
It is well known in the theory stochastic differential equations (SDE) that the mixing character of a random flow is closely related to the controllability properties of the associated deterministic dynamics. To be precise, let us consider the following SDE on a compact Riemannian manifold~$X$ without boundary:
\begin{equation} \label{0.1}
\dd u_t=V_0(u_t)\dd t+\sum_{j=1}^nV_j(u_t)\circ\dd\beta_j,\quad u_t\in X,
\end{equation}
where $V_0,V_1,\dots,V_n$ are smooth vector fields on~$X$, $\{\beta_j\}$ are independent Brownian motions, and the equation is understood in the sense of  Stratonovich. Along with~\eqref{0.1}, let us consider the controlled equation
\begin{equation} \label{0.3}
\dot u=V_0(u)+\sum_{j=1}^n\zeta^j(t)V_j(u), \quad u\in X.
\end{equation}
Here~$\zeta^j$ are real-valued piecewise continuous (control)  functions. Let us denote by $\Gamma(TX)$ the Lie algebra of smooth vector fields on~$X$ and by $\Lie(V_1,\dots,V_n)$ the minimal Lie subalgebra containing $V_j$, $j=1,\dots,n$.  We assume that the following conditions are fulfilled:
\begin{description}
\item[\bf H\"ormander condition.]
The subalgebra $\Lie(V_1,\dots,V_n)$ has full rank at some point $\hat u\in X$; that is, 
\begin{equation} \label{0.2}
\bigl\{V(\hat u) : V\in \Lie(V_1,\dots,V_n)\bigr\}=T_{\hat u}X,
\end{equation}
where~$T_u X$ stands for the tangent space of~$X$ at the point~$u$. 

\item[\bf Approximate controllability.]
For any $u_0, u_1\in X$ and any $\e>0$ there is $T>0$ and piecewise continuous functions $\zeta^j:[0,T]\to\R$ such that the solution $u(t)$ of~\eqref{0.3} issued from~$u_0$ belongs to the $\e$-neighbourhood\footnote{The manifold~$X$ is endowed with the natural distance associated with the Riemannian metric.} of~$u_1$ at time~$T$. 
\end{description}
Under the above hypotheses, the results established in~\cite{AK-1987} (see also~\cite{veretennikov-1988}) imply that the diffusion process generated by~\eqref{0.1} has a unique stationary measure. Thus, a sufficient condition for the uniqueness of a stationary distribution is expressed in terms of the control system~\eqref{0.3}: the first hypothesis is well known in the control theory and ensures the accessibility of~\eqref{0.3} (e.g., see Section~8.1 in~\cite{AS2004}), while the second is nothing else but the global approximate controllability in finite time. Let us remark that both papers mentioned above use the regularity of transition probabilities, and the latter is based on one or another form of the theory of hypoelliptic PDEs. 

The aim of our project is twofold: first, to investigate the problem of ergodicity for~\eqref{0.1} in the situation when the Brownian motions are replaced by other types of random processes (that need not to be Gaussian and therefore the tools related to hypoelliptic PDEs are not applicable), and second, to establish similar results for Markov processes corresponding to PDEs with a degenerate noise. The main emphasis is on the general principle according to which suitable controllability properties of the control system associated with the stochastic equation under study imply ergodicity of the latter. In this paper, we consider the situation where the convergence to the unique stationary measure holds in the total variation metric, and our main example is a differential equation driven by vector fields with random amplitudes. We refer the reader to Section~\ref{s1.2} for an exact formulation of our result on mixing and to Section~\ref{s2.1} for an application of it to ODEs on a compact manifold. 

\smallskip
Let us mention that the question of ergodicity for Markov processes is rather well understood, especially in the situation when the strong Feller property is satisfied; see the monographs~\cite{hasminski1980,nummelin1984,MT1993}. In the context of stochastic differential equations, the strong Feller property is often verified with the help of the Malliavin calculus or regularity of solutions for hypoelliptic PDEs; see Section~2.3 in~\cite{nualart1995}, Chapter~12 in~\cite{bogachev2010}, Section~11.5 in~\cite{daprato2014}, and Section~22.2 in~\cite{hormander2007}. In our approach, we do not use Malliavin calculus or the regularity theory for PDEs, replacing them by a general result on the image of probability measures under a smooth mapping that possesses a controllability property. Finally, let us emphasise that even though we confine ourselves to the case of a compact phase space, it is not difficult to extend the results to a more general setting of an unbounded space, assuming that the stochastic dynamics satisfies an appropriate dissipativity condition. 

\medskip
{\bf Acknowledgement.}
I am grateful to A.~Agrachev for numerous discussions on controllability properties of nonlinear systems and to S.~Kuksin  for suggesting a number of improvements. This research was carried out within the MME-DII Center of Excellence (ANR-11-LABX-0023-01) and supported by the RSF grant 14-49-00079.

\subsubsection*{Notation}
Let $X$ be a Polish space with a metric~$d$, let~$E$ be a separable Banach space, and let $J\subset\R$ be a bounded closed interval. We shall use the following notation. 

\smallskip
\noindent
$B_X(u,r)$ and $\dot B_X(u,r)$ denote, respectively, the closed and open ball in~$X$ of radius~$r$ centred at~$u$.

\smallskip
\noindent
$\BB(X)$ is the Borel $\sigma$-algebra on~$X$. 

\smallskip
\noindent
$C_b(X)$ is the space of bounded continuous functions $f:X\to\R$ with the norm 
$$
\|f\|_\infty=\sup_{u\in X}|f(u)|.
$$
In the case when $X$ is compact, we shall write $C(X)$. 

\noindent
$\PP(X)$ is the space of probability measures on~$X$. It is endowed with the total variation metric defined in Section~\ref{s1.1}. 

\smallskip
\noindent
$C(J,E)$ is the space of continuous functions $f:J\to E$ with the supremum norm. 

\smallskip
\noindent
$L^2(J,E)$ is the space of Borel-measurable functions $f:J\to E$ such that 
$$
\|f\|_{L^2(J,E)}=\biggl(\,\int_J\|f(t)\|_E^2\,\dd t\biggr)^{1/2}<\infty.
$$
In the case $E=\R$, we write $L^2(J)$.

\smallskip
\noindent
If $\varPhi:E\to F$ is a measurable mapping and $\mu\in\PP(E)$, then $\varPhi_*\mu$ denotes the image of~$\mu$ under~$\varPhi$. 

\smallskip
\noindent
For a set $\Gamma$, we denote by~$I_\Gamma$ its indicator function. If $f\in C_b(X)$ and $\mu\in\PP(X)$, then we write 
$$
(f,\mu)=\int_X f(u)\mu(\dd u). 
$$
In particular, we have $(I_\Gamma,\mu)=\mu(\Gamma)$. 

\smallskip
\noindent
$\DD(\xi)$ denotes the law of a random variable~$\xi$. 

\section{Mixing in terms of controllability properties}
\label{s1}
\subsection{General framework and definitions}
\label{s1.1}
Let~$(X,d)$ be a compact metric space, let~$E$ be a separable Banach space, and let $S:X\times E\to X$ be a continuous mapping. We consider the stochastic system
\begin{equation} \label{1.1}
u_k=S(u_{k-1},\eta_k), \quad k\ge1,
\end{equation}
supplemented with the initial condition
\begin{equation} \label{1.2}
u_0=u,
\end{equation}
where $\{\eta_k\}$ are i.i.d.\ $E$-valued random  variables and~$u$ is a random variable in~$X$ independent of~$\{\eta_k\}$. In this case, the trajectories of~\eqref{1.1} form a discrete-time Markov process~$(u_k,\IP_u)$, and we denote by $P_k(u,\Gamma)$ its transition function and by~$\{\PPPP_k\}$ and~$\{\PPPP_k^*\}$ the corresponding Markov semigroups acting in the spaces~$C(X)$ and~$\PP(X)$, respectively. Recall that a measure $\mu\in\PP(X)$ is said to be {\it stationary\/} for~$(u_k,\IP_u)$ if $\PPPP_1^*\mu=\mu$. Our aim in this section is to establish a sufficient condition for the uniqueness of a stationary measure and its (exponential) stability in the total variation metric
$$
\|\mu_1-\mu_2\|_{\mathrm{var}}
=\sup_{\Gamma\in\BB(X)}|\mu_1(\Gamma)-\mu_2(\Gamma)|
=\frac12\sup_{\|f\|_{\infty}\le1}|(f,\mu_1)-(f,\mu_2)|,
$$
where the second supremum is taken over all continuous functions whose $L^\infty$ norm is bounded by~$1$. 

\smallskip
Let us introduce some controllability properties associated with the stochastic system~\eqref{1.1}. 

\medskip
\underline{\it Approximate controllability to a given point}.
Given any initial point  $\hat u\in X$, we say that~\eqref{1.1} is {\it globally approximately controllable to~$\hat u$\/} if for any $\e>0$ there is a compact set $\KK=\KK_\e\subset E$ and an integer $m=m_\e\ge1$ such that, given an initial point $u\in X$, one can find $\zeta_1,\dots,\zeta_m\in\KK$ satisfying the inequality
\begin{equation} \label{1.3}
d\bigl(S_m(u;\zeta_1,\dots,\zeta_m),\hat u\bigr)\le\e,
\end{equation}
where $S_k(u;\eta_1,\dots,\eta_k)$ stands for the trajectory of~\eqref{1.1}, \eqref{1.2}. 

\medskip
\underline{\it Solid controllability}.
Following~\cite{AS-2005} (see Section~12), we say that~\eqref{1.1} is {\it solidly controllable from~$\hat u$\/} if there is a compact set $Q\subset E$, a non-degenerate ball $B\subset X$, and a number~$\e>0$ such that, for any continuous mapping $\varPhi:Q\to X$ satisfying the condition
\begin{equation} \label{1.4}
\sup_{\zeta\in Q}d\bigl(\varPhi(\zeta),S(\hat u,\zeta)\bigr)\le\e,
\end{equation}
we have $\varPhi(Q)\supset B$. 

\medskip
We shall also need a class of probability measures on~$E$. A measure $\ell\in\PP(E)$  is said to be {\it decomposable\/} if there are two sequences of closed subspaces~$\{F_n\}$ and~$\{G_n\}$ in~$E$ such that the following properties hold:
\begin{itemize}
\item[\bf(i)]
we have $\dim F_n<\infty$ and $F_n\subset F_{n+1}$ for any $n\ge1$, and the union $\cup_nF_n$ is dense in~$E$.
\item[\bf(ii)]
the space~$E$ can be represented as the direct sum of~$F_n$ and~$G_n$, the operator norms of the corresponding projections~${\mathsf P}_n$ and~${\mathsf Q}_n$ are bounded, and for any $n\ge1$ the measure~$\ell$ can be written as the product of its projections ${\mathsf P}_{n*}\ell$ and~${\mathsf Q}_{n*}\ell$. 
\end{itemize}
Note that the boundedness of~${\mathsf P}_n$ is equivalent to the following property:
\begin{equation} \label{1.9}
{\mathsf P}_n\to I, \quad {\mathsf Q}_n\to0\quad
\mbox{in the strong operator topology}. 
\end{equation}
Indeed, the fact that~\eqref{1.9} implies the boundedness of the norms of~${\mathsf P}_n$ and~${\mathsf Q}_n$ follows immediately from Baire's theorem. Conversely, suppose that the norms of~${\mathsf P}_n$ are bounded by a number~$C$ and fix $\zeta\in E$. In view of the density of $\cup_nF_n$, there are $\zeta_n\in F_n$ such that $\|\zeta-\zeta_n\|_E\to0$ as $n\to\infty$. It follows that 
\begin{align*}
\|\zeta-{\mathsf P}_n\zeta\|_E
&\le \|\zeta-\zeta_n\|_E+\|{\mathsf P}_n(\zeta-\zeta_n)\|_E
+\|\zeta_n-{\mathsf P}_n\zeta_n\|_E\\
&\le (C+1) \|\zeta-\zeta_n\|_E,
\end{align*}
where we used the relation ${\mathsf P}_n\zeta_n=\zeta_n$. This implies that ${\mathsf P}_n\zeta\to\zeta$ as $n\to\infty$.

\subsection{Exponential mixing in the total variation metric}
\label{s1.2}
Recall that we  consider the stochastic system~\eqref{1.1}, in which $S:X\times E\to X$ is a continuous mapping and~$\{\eta_k\}$ is a sequence of i.i.d.\ random variables in~$E$ whose law~$\ell$ is a decomposable measure on~$E$. We shall denote by~$\ell_n$ the image of~$\ell$ under the  projection to the subspace~$F_n$ (entering the definition of a decomposable measure). We shall say that a stationary measure~$\mu\in\PP(X)$ for 
the Markov process~$(u_k,\IP_u)$ associated with~\eqref{1.1} is {\it exponentially mixing\/} if there are positive numbers~$\gamma$ and~$C$ such that
\begin{equation} \label{4.3}
\|\PPPP_k^*\lambda-\mu\|_{\mathrm{var}}\le Ce^{-\gamma k}
\quad\mbox{for $k\ge0$, $\lambda\in\PP(X)$}. 
\end{equation}

\begin{theorem} \label{t1.1}
Let us assume that $X$ is a compact Riemannian manifold, the mapping $S(u,\cdot):E\to X$ is infinitely differentiable in the Fr\'echet sense, and its derivative~$(D_\eta S)(u,\eta)$ is a continuous function of~$(u,\eta)$. Suppose, in addition, that~\eqref{1.1} is globally approximately controllable to a point~$\hat u\in X$  and is solidly controllable\,\footnote{The importance of the concept of {\it solid controllability\/} was first noted by Agrachev and Sarychev in~\cite{AS-2005} (see also~\cite{AS-2008}). It was later used in~\cite{AKSS-aihp2007} to establish absolute continuity of finite-dimensional projections of laws for solutions of stochastic PDEs.} from~$\hat u$, the law~$\ell$ of the random variables~$\eta_k$ is decomposable, and the measures~${\mathsf P}_{n*}\ell$ possess  positive continuous densities~$\rho_n$ with respect to the Lebesgue measure on~$F_n$. Then the Markov process~$(u_k,\IP_u)$ associated with~\eqref{1.1} has a unique stationary measure~$\mu\in\PP(X)$, which is exponentially mixing.
\end{theorem}

\begin{proof}
We shall prove that the hypotheses of Theorem~\ref{t4.1} (see Appendix) are fulfilled for the Markov process~$(u_k,\IP_u)$. Thus, we need to show that~\eqref{4.1} and~\eqref{4.2} hold for some positive numbers~$\e$, $\delta$, $p$, and~$m$. 

\smallskip
{\it Step~1: Recurrence}. 
The global approximate controllability will immediately imply~\eqref{4.1} if we prove that the support of~$\ell$ coincides with~$E$. Indeed, let us fix any $\delta>0$. In view of global approximate controllability, there exist an integer $m\ge1$ and a compact set $\KK\subset E$ such that, given $u\in X$, one can find vectors $\zeta_1^u,\dots,\zeta_m^u\in\KK$ such that 
$$
d_X(S_m(u;\zeta_1^u,\dots,\zeta_m^u),\hat u)\le\delta/2. 
$$
By the uniform continuity of~$S_m$ on the compact set $X\times \KK^m$ (where~$\KK^m$ stands for the $m$-fold product of the set~$\KK$ with itself), one can find $\e>0$ not depending on~$u$ such that
$S_m(u;\zeta_1,\dots,\zeta_m)\in \dot B_X(\hat u,\delta)$
for any vectors $\zeta_1,\dots,\zeta_m\in E$ satisfying the inequalities $\|\zeta_k-\zeta_k^u\|_E\le\e$ with $1\le k\le m$. What has been said implies that 
$$
P_m(u,\dot B_X(\hat u,\delta))\ge \IP\bigl\{\eta_k\in B_E(\zeta_k^u,\e), 1\le k\le m\bigr\}
=\prod_{k=1}^m\IP\{\eta_1\in B_E(\zeta_k^u,\e)\}. 
$$
The product on the right-hand side of this inequality is positive because the support of the law of~$\eta_1$ coincides with~$E$. It follows from the portmanteau theorem (see Theorem~11.1.1 in~\cite{dudley2002}) that the function $u\mapsto P_m(u,\dot B_X(\hat u,\delta))$ defined on the compact space~$X$ is lower-semicontinuous and, hence, minorised by a positive number~$p$. This implies the required inequality~\eqref{4.1}. 

\smallskip
We now prove that $\supp\ell=E$. Since the union $\cup_nF_n$ is dense in~$E$, the latter property will be established once we have shown that $\supp\DD(\eta_1)\supset F_m$ for any $m\ge1$. Let us fix any integer~$m\ge1$, a vector $\hat\eta\in F_m$, and a number~$\e>0$. It follows from~\eqref{1.9} that the sequence~$\{\mathsf Q_n\eta_1\}$ goes to zero almost surely and therefore also in probability. Hence, 
\begin{equation} \label{1.8}
\IP\{\|\mathsf Q_n\eta_1\|>\e/2\}\to0\quad\mbox{as $n\to\infty$}. 
\end{equation}
For any $n\ge m$, we write
\begin{align*}
\IP\{\eta_1\in B_E(\hat\eta,\e)\}&=\IP\{\|\eta_1-\hat\eta\|\le\e\}\\
&\ge \IP\{\|{\mathsf P}_n\eta_1-\hat\eta\|\le\e/2,\|{\mathsf Q}_n\eta_1\|\le\e/2\}\\
&= \IP\{\|{\mathsf P}_n\eta_1-\hat\eta\|\le\e/2\}\,\IP\{\|{\mathsf Q}_n\eta_1\|\le\e/2\}. 
\end{align*}
The first factor on the right-hand side is positive, since $\hat\eta\in\supp\DD({\mathsf P}_n\eta_1)$ due to the positivity of the density~$\rho_n$. In view of~\eqref{1.8}, the second factor goes to one as $n\to\infty$, so that the right-hand side is positive for sufficiently large~$n$. Since~$\e>0$ was arbitrary, we conclude that $\hat\eta\in\DD(\eta_1)$. 

\smallskip
{\it Step~2: Coupling}. 
We need to prove inequality~\eqref{4.2}. To this end, we first establish a lower bound for the measures~$P_1(u,\cdot)$ on a ball $B_X(\hat u,\delta)$, where $\delta>0$ is sufficiently small. This will be done with the help of Proposition~\ref{p4.3}. 

By the hypothesis, \eqref{1.1} is solidly controllable from~$\hat u$. We denote by~$Q\subset E$ a compact subset such that the image of any mapping~$\varPhi:Q\to X$ satisfying inequality~\eqref{1.4} with $\e\ll1$ contains a ball in~$X$. It follows from~\eqref{1.9} that 
$$
\sup_{\zeta\in Q}\|{\mathsf P}_n\zeta-\zeta\|_E\to0\quad\mbox{as $n\to\infty$}. 
$$ 
Combining this with the uniform continuity of $S(u,\cdot):E\to X$ on~$Q$, we see that~\eqref{1.4} is satisfied for~$\varPhi(\zeta)=S(\hat u,{\mathsf P}_n\zeta)$ with a sufficiently large $n\ge1$. Thus, there is an integer $n\ge1$ and a ball $Q_1\subset F_n$ such that the image of the mapping $S(\hat u,\cdot):Q_1\to X$ covers a ball in~$X$. By the Sard theorem (see Section~II.3 in~\cite{sternberg1983}), there is $\hat\zeta\in Q_1$ such that the derivative $(D_\eta S)(\hat u,\hat\zeta)$ has a full rank. Proposition~\ref{p4.3} implies that there is $\delta>0$ and a continuous function $\psi:B_X(\hat u,\delta)\times X\to\R_+$ such that
\begin{gather}
\psi(\hat u,\hat x)>0,\label{1.10}\\
S(u,\cdot)_*\ell\ge \psi(u,x)\vol(\dd x)\quad\mbox{for $u\in B_X(\hat u,\delta)$},
\label{1.11}
\end{gather}
where $\hat x=S(\hat u,\hat\zeta)$, and $\vol(\cdot)$ denotes the Riemannian measure on~$X$. It follows from~\eqref{1.10} that taking, if necessary, a smaller~$\delta>0$, we obtain
\begin{equation} \label{1.12}
\psi(u,x)\ge\e >0\quad
\mbox{for $u\in B_X(\hat u,\delta)$, $x\in B_X(\hat x,\delta)$}. 
\end{equation}

Now note that $S_*(u,\ell)=P_1(u,\cdot)$. Combining~\eqref{1.11} and~\eqref{1.12}, we derive 
$$
P_1(u,\dd x)\ge\e I_{B_X(\hat x,\delta)}(x)\vol(\dd x)\quad
\mbox{for $u\in B_X(\hat u,\delta)$},
$$
where $I_\Gamma$ stands for the indicator function of~$\Gamma$.
It follows that
$$
\|P_1(u,\cdot)-P_1(u',\cdot)\|_{\mathrm{var}}\le 1-\e\vol(B_X(\hat x,\delta)).
$$
This completes the proof of Theorem~\ref{t1.1}. 
\end{proof}

\section{Differential equations on a compact manifold}
\label{s2}

\subsection{Main result}
\label{s2.1}
Let~$X$ be a compact Riemannian manifold of dimension $d\ge1$ without boundary. We consider the ordinary differential equation
\begin{equation} \label{2.1}
\dot u=V_0(u)+\sum_{j=1}^n\eta^j(t)V_j(u), \quad u(t)\in X. 
\end{equation}
Here $V_j$, $j=0,\dots,n$, are smooth vector fields on~$X$ and~$\eta^j(t)$ are real-valued random processes of the form
\begin{equation} \label{2.3}
\eta^j(t)=\sum_{k=1}^\infty I_{[k-1,k)}(t)\eta_k^j(t-k+1),
\end{equation}
where $\eta_k^j$ are random variables in $L^2(J)$ with $J=[0,1]$ such that the vector functions $\eeta_k=(\eta_k^1,\dots,\eta_k^n)$ are i.i.d.\ random variables in $E:=L^2(J,\R^n)$. We denote by~$\ell\in\PP(E)$ the law of~$\eeta_k$, $k\ge1$.

Before formulating the main result of this section, we recall some well-known facts about Eq.~\eqref{2.1}. Let~$\eta^j:\R_+\to\R^n$ be measurable functions that are integrable on any compact subset of~$\R_+$. Then, for any $v\in X$, there is a unique absolutely continuous function $u:\R_+\to X$ that satisfies Eq.~\eqref{2.1} for almost every $t\ge0$ and the initial condition
\begin{equation} \label{2.2}
u(0)=v. 
\end{equation}
Moreover, if we denote by~$S$ a mapping that acts from~$X\times E$ to~$X$ and takes the pair $(v,\eeta)$ to~$u(1)$, where $u(t)$ is the solution of problem~\eqref{2.1}--\eqref{2.2} on~$J$ with $(\eta^1,\dots,\eta^n)=\eeta$, then classical results from the theory of ordinary differential equations imply that~$S$ is infinitely differentiable in the Fr\'echet sense. We denote $u_k=u(k)$ and observe  that
\begin{equation} \label{2.40}
u_k=S(u_{k-1},\eeta_k), \quad k\ge1. 
\end{equation}
Since the random variables~$\{\eeta_k\}$ are i.i.d., the family of all sequences~$\{u_k\}$ satisfying~\eqref{2.40} form a discrete-time Markov process, which is denoted by~$(u_k,\IP_u)$. We write~$\PPPP_k$ and~$\PPPP_k^*$ for the corresponding Markov semigroups. 

We say that the control system~\eqref{0.3} considered on~$X$ satisfies the {\it weak H\"ormander condition\/} at a point $\hat u\in X$ if there are~$d$ vector fields in the family 
$$
\{V_j, j=1,\dots,n; [V_j,V_k], 0\le j,k\le n; [[V_j,V_k],V_l], 0\le j,k,l\le n, \dots\}
$$
that are linearly independent at the point~$\hat u$. In other words, denoting
\begin{equation} \label{2.04}
V_\zzeta=V_0+\zeta^1V_1+\cdots+\zeta^n V_n\quad\mbox{for $\zzeta=(\zeta^1,\dots,\zeta^n)\in\R^n$},
\end{equation}
the weak H\"ormander condition is equivalent to the hypothesis that zero-time ideal\footnote{We do not use this concept in what follows, so the reader not familiar with it may safely ignore this reformulation.} of the family $\{V_\zzeta,\zzeta\in\R^n\}$ has full rank at~$\hat u$; see Section~2.4 in~\cite{jurdjevic1997}.
We refer the reader to Section~2.3 in~\cite{nualart1995} and Section~2 in~\cite{hairer-2011} for a discussion of this condition from the probabilistic point of view. 

Let us set $\XXX=C(J,X)$. The theorem below proved in the next subsection describes the large-time asymptotics of the laws of trajectories for~\eqref{2.1}--\eqref{2.2}. 

\begin{theorem} \label{t2.1}
In addition to the above hypotheses, assume that the following two conditions are satisfied:

\begin{itemize}
\item[\bf(a)]
There is a point $\hat u\in X$ 
such that system~\eqref{1.1} is globally approximately controllable to~$\hat u$, and the weak H\"ormander condition holds at~$\hat u$.
\item[\bf(b)]
The law~$\ell$ is decomposable, and the measures\,\footnote{We denote by~$F_n$ the finite-dimensional spaces entering the definition of a decomposable measure and by~${\mathsf P}_n$ the corresponding projections.}~${\mathsf P}_{n*}\ell$ possess positive continuous densities~$\rho_n$ with respect to the Lebesgue measure on~$F_n$.
\end{itemize}
Then there is a unique measure $\mmu\in\PP(\XXX)$ and positive numbers~$\gamma$ and~$C$ such that, for any $X$-valued random variable~$v$ independent of~$\{\eeta_k\}$, the solution~$u(t)$ of~\eqref{2.1}--\eqref{2.2} satisfies the inequality 
\begin{equation} \label{2.50}
\|\DD(\uuu_k)-\mmu\|_{\mathrm{var}}\le Ce^{-\gamma k}, 
\quad k\ge1,
\end{equation}
where $\uuu_k$ stands for the restriction of~$u(t)$ to the interval $[k-1,k]$, and $\|\cdot\|_{\mathrm{var}}$ denotes the total variation norm on~$\PP(\XXX)$. 
\end{theorem}

\subsection{Proof of Theorem~\ref{t2.1}}
\label{s2.2}
We begin with a simple remark reducing the proof of theorem to the problem of exponential mixing for the discrete-time Markov process~$(u_k,\IP_u)$ associated with~\eqref{2.40}. Suppose we have proven that~\eqref{2.40} has a unique stationary measure $\mu\in\PP(X)$, which exponentially mixing in the sense that~\eqref{4.3} holds for the corresponding Markov semigroup. Let us denote by $\SSS:X\times E\to\XXX$ a mapping that takes~$(v,\eeta)$ to $(u(t),t\in J)$, where $u(t)$ is the solution of problem~\eqref{2.1}--\eqref{2.2} on~$J$ with $(\eta^1,\dots,\eta^n)=\eeta$. It follows from the independence of~$\{\eeta_k\}$ that 
\begin{equation} \label{2.51}
\DD(\uuu_k)=\SSS_*\bigl((\PPPP_{k-1}^*\lambda)\otimes\ell\bigr)
\quad\mbox{for any $k\ge1$},
\end{equation}
where $\lambda=\DD(v)$. Let us set  $\mmu=\SSS_*(\mu\otimes\ell)$. Since the total variation distance does not increase under a measurable mapping, relation~\eqref{2.51} implies  that 
\begin{equation} \label{2.52}
\|\DD(\uuu_k)-\mmu\|_{\mathrm{var}}
\le \|(\PPPP_{k-1}^*\lambda)\otimes\ell
-\mu\otimes\ell\|_{\mathrm{var}}
=\|\PPPP_{k-1}^*\lambda-\mu\|_{\mathrm{var}},
\quad k\ge1.
\end{equation}
The required inequality~\eqref{2.50} follows from~\eqref{4.3} and~\eqref{2.52}.  

\medskip
We thus need to prove that $(u_k,\IP_u)$ has a unique stationary measure, which is exponentially mixing in the total variation metric. To this end, we show that the hypotheses of Theorem~\ref{t1.1} are fulfilled. Namely, it suffices to check that the mapping $S(u,\cdot):E\to X$ is infinitely differentiable, $(D_\eta S)(u,\eta)$ is continuous on $X\times E$, and~\eqref{1.1} is solidly controllable from~$\hat u$. As was mentioned above, the first two properties are true due to classical results in the theory of ordinary differential equations. To prove the solidly controllability from~$\hat u$, we use a degree theory argument (see Section~12.2 in~\cite{AS-2005} and Section~2.3 in~\cite{shirikyan-aihp2007}) and a well-known idea from the control theory (see the proof of Theorem~3 in Section~1.2 of~\cite[Chapter~3]{jurdjevic1997}). 

\smallskip
{\it Step~1: Reduction to continuous exact controllability\/}. 
Given a closed ball $B=B_X(\hat v,r)$, we shall say that~\eqref{1.1} is {\it continuously exactly controllable from~$\hat u$ to~$B$\/} if there is  a continuous mapping $f:B\to E$ such that 
\begin{equation} \label{2.9}
S(\hat u,f(v))=v\quad\mbox{for any $v\in B$}. 
\end{equation}
We claim that if~\eqref{1.1} is continuously exactly controllable from~$\hat u$ to some ball $B=B_X(\hat v,r)$, then it is solidly controllable from~$\hat u$. Indeed, given a continuous mapping $\varphi:B\to  X$ and a point $z\in X\setminus\varphi(\p B)$, we denote by $\deg(\varphi,B,z)$ the degree of~$\varphi$ at~$z$. Let us choose~$\e>0$ so small that 
\begin{equation} \label{2.10}
\deg(\varphi,B,z)=\deg(I,B,z)=1\quad \mbox{for any $z\in B_X(\hat v,\e)$},
\end{equation}
where $I:B\to B$ is the identity mapping and $\varphi:B\to X$ is an arbitrary continuous mapping such that 
\begin{equation} \label{2.11}
\sup_{v\in B}d_X(\varphi(v),v)\le\e.
\end{equation}
Denote $Q=f(B)$, where~$f$ is the mapping entering~\eqref{2.9}, and consider any continuous mapping $\varPhi:Q\to X$ satisfying~\eqref{1.4}. Then  inequality~\eqref{2.11} is true for $\varphi=\varPhi\circ f$, whence it follows that~\eqref{2.10} holds. In particular, for any $z\in B_X(\hat v,\e)$ there is $y\in B$ such that $\varPhi(f(y))=z$. We have thus shown that $\varPhi(Q)\supset B_X(\hat v,\e)$.

\smallskip
{\it Step~2: Extended system\/}. 
To prove the continuous exact controllability of~\eqref{1.1}, let us introduce the extended phase space $\widetilde X=X\times\R$, with a natural Riemannian structure, and consider the new control system
\begin{equation} \label{2.4}
\dot y=\widetilde V_0(y)+\sum_{j=1}^n\zeta^j(t)\widetilde V_j(y), \quad y\in\widetilde X,
\end{equation}
where $y=(u,z)$, $\widetilde V_0(y)=(V_0(u),1)$, and $\widetilde V_j(y)=(V_j(u),0)$ for $j=1,\dots,n$. It is clear that if~$y(t)$ is a trajectory for~\eqref{2.4}, then the projection of~$y$ to~$X$ is a trajectory for~\eqref{0.3}, and vice versa, any trajectory of~\eqref{0.3} can be extended to a trajectory of~\eqref{2.4} by adding to it the function $z(t)=t+z_0$, where $z_0\in\R$ is an arbitrary initial point. In what follows, given a vector $\zzeta=(\zeta^1,\dots,\zeta^n)\in\R^n$, we shall write 
$$
{\widetilde V}_\zzeta
={\widetilde V}_0+\zeta^1{\widetilde V}_1+\cdots+\zeta^n {\widetilde V}_n,
\quad V_\zzeta
=V_0+\zeta^1V_1+\cdots+\zeta^n V_n
$$
and denote $\VV=\{V_\zzeta,\zzeta\in\R^n\}$ and $\widetilde\VV=\{{\widetilde V}_\zzeta,\zzeta\in\R^n\}$. Notice that the last component of~$\widetilde V_\zzeta$ is equal to~$1$ for any $\zzeta\in\R^n$. 

\smallskip
{\it Step~3: Lie algebra generated by~$\widetilde\VV$\/}. 
Let us denote by~$\Lie(\widetilde \VV)$ the Lie algebra generated by~$\widetilde\VV$.
We claim that~$\Lie(\widetilde \VV)$ has full rank at any point $(\hat u,z)$ with $z\in\R$; that is, the space of restrictions of the vector fields from $\Lie(\widetilde \VV)$ to~$(\hat u,z)$ coincide with the tangent space~$T_{(\hat u,z)}\widetilde X$. Indeed, it is straightforward to check that
$$
[{\widetilde V}_i,{\widetilde V}_j]=([V_i,V_j],0)\quad\mbox{for $0\le i,j\le n$},
$$
whence it follows that the derived algebra\footnote{Recall that the {\it derived algebra\/} of a family of vector fields~$\WW$ is defined as the vector span of all possible (iterated) commutators of the elements of~$\WW$.} of~$\widetilde\VV$ has the form
\begin{equation} \label{2.6}
\widetilde\DD=\{(W,0),W\in\DD\},
\end{equation}
where~$\DD$ stands for the derived algebra of~$\VV$. The weak H\"ormander condition implies that 
$$
\lspan\{V_1,\dots,V_n,\DD\}\bigr|_{\hat u}=T_{\hat u}X. 
$$
Combining this with~\eqref{2.6}, we see that
$$
\lspan\{{\widetilde V}_1,\dots,{\widetilde V}_n,{\widetilde \DD}\}\bigr|_{(\hat u,z)}=T_{\hat u}X\times\{0\}. 
$$
Recalling that ${\widetilde V}_0=(V_0,1)$, we obtain the required result.  

\smallskip
{\it Step~4: Continuous exact controllability at a time $\tau\in(0,1)$\/}. 
Given an interval $J_\tau=[0,\tau]$ and a function $\zzeta\in L^2(J_\tau,\R^n)$, we denote by~$\RR_\tau(\zzeta)$ the value at time~$\tau$ of the solution of~\eqref{0.3} issued from~$\hat u$.  We claim that there is $\tau\in(0,1)$, a closed ball $B'\subset X$, and a continuous function $g:B'\to L^2(J_\tau,\R^n)$ such that 
\begin{equation} \label{2.12}
\RR_\tau(g(v))=v\quad\mbox{for any $v\in B'$}. 
\end{equation}
To prove this, consider the extended system~\eqref{2.4} and, for $\zzeta\in\R^n$ and  $y_0\in\widetilde X$, denote by $e^{t\widetilde V_\zzeta}y_0$ its solution issued from~$y_0$ and corresponding to the control functions~$(\zeta^1,\dots,\zeta^n)\equiv\zzeta$. Suppose we have found vectors $\zzeta_0,\dots,\zzeta_d\in\R^n$ and an open parallelepiped 
$$
\widetilde\Pi=\{\aalpha=(\alpha_0,\dots,\alpha_d)\in\R^{d+1}:
a_l<\alpha_l<b_l\mbox{ for }0\le l\le d\}\subset [0,1]^{d+1}
$$
such that $\sum_lb_l<1$, and the mapping 
\begin{equation} \label{5.1}
\widetilde F:\widetilde\Pi\to\widetilde X, \quad 
\aalpha\mapsto e^{\alpha_d\widetilde V_{\zzeta_d}}\circ\cdots\circ
e^{\alpha_0\widetilde V_{\zzeta_0}}(\hat u,0)
\end{equation}
is an embedding of~$\widetilde\Pi$ into~$\widetilde X$. For any $\aalpha\in\widetilde\Pi$, we set $T_\aalpha=\alpha_0+\cdots+\alpha_d$ and define $\zzeta^\aalpha:[0,T_\aalpha]\to\R^n$ by the relation
$$
\zzeta^\aalpha(t)=\zzeta_l\quad
\mbox{for $\alpha_0+\cdots+\alpha_{l-1}\le t<\alpha_0+\cdots+\alpha_{l}$}, 
$$
where $l=0,\dots,d$, and the left-hand bound in the inequality is taken to be zero for $l=0$. Then, denoting by~$F$ to projection of~$\widetilde F$ to~$X$, we see that
\begin{equation} \label{5.2}
F(\aalpha)=\RR_{T_\aalpha}(\zzeta^\aalpha)\quad\mbox{for $\aalpha\in\widetilde\Pi$}. 
\end{equation}
We now fix $\hat\aalpha\in\widetilde\Pi$ and denote by~$\Pi$ the intersection of~$\widetilde\Pi$ with the $d$-dimensional hyperplane $L_{\hat\aalpha}=\{\alpha_0+\cdots+\alpha_{d}=\tau\}\subset\R^{d+1}$, where $\tau =T_{\hat \aalpha}$. Then~$\Pi$ is an open polyhedron in~$L_{\hat\aalpha}$. Since the last component of~$V_\zzeta$ is equal to~$1$ for any $\zzeta\in\R^n$, the last component of~$\widetilde F(\aalpha)$ is equal to~$\tau$ for any $\aalpha\in\Pi$, so that $\widetilde F(\widetilde\Pi)$ lies in the set $\{(u,z)\in\widetilde X:z=\tau\}$. Combining this fact with~\eqref{5.2}, we see that $\RR_\tau:\Pi\to X$ is a diffeomorphism of~$\Pi$ onto its image. Denote by~$\RR_\tau^{-1}$ its inverse. Now let $B'\subset \RR_\tau(\Pi)$ be an arbitrary closed ball. Then the mapping 
$$
g:B'\to L^2(J_\tau,\R^n), \quad g(v)=\zzeta^{\RR_\tau^{-1}(v)},
$$
is continuous and satisfies the required relation~\eqref{2.12}. 

Thus, it remains to find a parallelepiped $\widetilde\Pi$ such that~$\widetilde F$ defined by~\eqref{5.1} is an embedding. Even though this is a well-known result,  for the reader's convenience, we outline the main idea, following the argument in the proof of Krener's theorem (e.g., see Theorem~8.1 in~\cite{AS2004}).

It  was proved in Step~3 that $\Lie(\widetilde \VV)$ has full rank at the point $\hat y=(\hat u,0)$. By continuity, there is an open set $U\subset X$ containing~$\hat y$ such that $\Lie(\widetilde \VV)$ has full rank at any $y\in U$. In the construction below, we assume, without mentioning it explicitly, that all the points belong to~$U$. We shall construct vectors $\zzeta_j\in\R^n$, $0\le j\le d$, and numbers $0<a_j<b_j<1$ such that $\sum_jb_j<1$, and the following properties hold: 
\begin{itemize}
\item[(i)] 
The mapping $F_j:(\alpha_0,\dots,\alpha_j)\mapsto 
e^{\alpha_j\widetilde V_{\zzeta_j}}\circ\cdots\circ
e^{\alpha_0\widetilde V_{\zzeta_0}}(\hat u,0)$ defines an embedding of the open parallelepiped 
$$
\Pi_j
=\{(\alpha_0,\dots,\alpha_j)\in\R^{j+1}:
a_l<\alpha_l<b_l\mbox{ for }0\le l\le j\}\subset[0,1]^{j+1}
$$
into the manifold~$\widetilde X$; we denote by~$Y_j$ the image of~$\Pi_j$ under~$F_j$. 
\item[(ii)] 
The vector field $\widetilde V_{\zzeta_{j}}(y)$ is transversal to~$Y_{j-1}$ at any point $y\in Y_{j-1}$. 
\end{itemize}
Once this is established, one can take $\widetilde\Pi=\Pi_d$, completing thus the construction of~$g$. To prove the above properties, we proceed by recurrence. For $j=0$, we take any $\zzeta_0\in\R^n$ such that~$\widetilde V_{\zzeta_0}(\hat y)\ne0$. We then set $a_0=0$ and choose $b_0\in(0,1)$ so small that $F_0(\alpha)$ is an embedding of~$\Pi_0$. Property~(ii) is trivial for $j=0$. 

Let us assume that the vectors $\zzeta_l\in\R^n$ and the intervals $(a_l,b_l)$ have been  constructed for $0\le l\le j-1$. Since $\Lie(\widetilde \VV)$ has full rank at any point $y\in Y_{j-1}$, we can find $\zzeta_j\in\R^n$ and $y_j\in Y_{j-1}$ such that $\widetilde V_{\zzeta_j}(y_j)$ is transversal to~$Y_{j-1}$. By continuity, reducing the size of the intervals $(a_l,b_l)$ if necessary, we can assume that $\widetilde V_{\zzeta_j}(y)$ is transversal to~$Y_{j-1}$ at any point $y\in Y_{j-1}$. We now set $a_j=0$ and choose $b_j>0$ so small that $b_0+\cdots+b_j<1$ and~$F_j(\alpha_0,\dots,\alpha_j)$ defines an embedding of~$\Pi_j$ into~$\widetilde X$. We have thus established the required property.

\smallskip
{\it Step~5: Completion of the proof\/}. 
We can now easily prove the validity of inequality~\eqref{2.9}, in which $B\subset X$ is a closed ball. To this end, we define $\psi:X\to X$ as the mapping that takes $w_0\in X$ to~$w(1-\tau)$, where~$w(t)$ the solution of the equation $\dot w=V_0(w)$ issued from~$w_0$. It is well known from the theory of ordinary differential equations that~$\psi$ is a diffeomorphism of~$X$. Given any $v\in B'$, we extend~$g(v)\in L^2(J_\tau,\R^n)$ to the interval~$(\tau,1]$  by zero and note that, in view of~\eqref{2.12}, we have 
$$
S(\hat u,g(v))=\RR_1(g(v))=(\psi\circ\RR_\tau)(g(v))=\psi(v)
\quad\mbox{for $v\in B'$}.
$$ 
Defining $f:\psi(B')\to L^2(J_\tau,\R^n)$ by the relation $f(v)=g(\psi^{-1}(v))$, we see that
$$
S(\hat u, f(v))=v\quad\mbox{for $v\in\psi(B')$}.
$$
It remains to note that since~$\psi$ is a diffeomorphism, the set~$\psi(B')$ contains a non-degenerate closed ball $B\subset X$, and hence~\eqref{2.9} holds. This completes the proof of Theorem~\ref{t2.1}. 

\section{Appendix}
\label{s4}
\subsection{Sufficient condition for mixing}
\label{s4.1}
Let $X$ be a compact metric space and let $(u_k,\IP_u)$ be a discrete-time Markov process in~$X$. Since~$X$ is compact, $(u_k,\IP_u)$ has at least one stationary measure~$\mu\in\PP(X)$. The following result gives a sufficient condition for the uniqueness of stationary measure and its exponential stability.

\begin{theorem} \label{t4.1}
Suppose there is a point $\hat u\in X$ and a number~$\delta>0$ such that the following conditions are satisfied.

\noindent
{\bf Recurrence:} 
There is $p>0$ and an integer $m\ge1$ such that
\begin{equation} \label{4.1}
P_m\bigl(u,B_X(\hat u,\delta)\bigr)\ge p\quad\mbox{for any $u\in X$}.
\end{equation}

\noindent
{\bf Coupling:}
There is $\e>0$ such that
\begin{equation} \label{4.2}
\|P_1(u,\cdot)-P_1(u',\cdot)\|_{\mathrm{var}}\le 1-\e\quad
\mbox{for any $u,u'\in B_X(\hat u,\delta)$}. 
\end{equation}
Then $(u_k,\IP_u)$ has a unique stationary measure $\mu\in\PP(X)$, which is exponentially mixing for the total variation metric in the sense that~\eqref{4.3} holds for some positive numbers~$\gamma$ and~$C$.
\end{theorem}

Even though this theorem is a particular case of more general results established in~\cite[Chapters~15 and~16]{MT1993} (see also~\cite{HM-bsm2011}), we give a direct proof of it for the reader's convenience. 

\begin{proof}
We shall prove that the mapping $\PPPP_{m+1}^*:\PP(X)\to\PP(X)$ is a contraction. This will imply all required results. 

\medskip
{\it Step 1\/}. 
Let us recall that, given two measures $\lambda,\lambda'\in\PP(X)$, we can find $\nu,\hat\lambda,\hat\lambda'\in\PP(X)$ such that (e.\,g., see Corollary~1.2.25 in~\cite{KS-book})
\begin{equation} \label{4.03}
\lambda=(1-d)\nu+d\hat\lambda,\quad \lambda'=(1-d)\nu+d\hat\lambda',
\end{equation}
where $d=\|\lambda-\lambda'\|_{\rm{var}}$. It follows that 
$$
\bigl\|\PPPP_{m+1}^*\lambda-\PPPP_{m+1}^*\lambda'\bigr\|_{\rm{var}}
=d\,\bigl\|\PPPP_{m+1}^*\hat\lambda-\PPPP_{m+1}^*\hat\lambda'\bigr\|_{\rm{var}}
=d\,\bigl\|\PPPP_1^*\mu-\PPPP_1^*\mu'\bigr\|_{\rm{var}},
$$
where we set $\mu=\PPPP_m^*\hat\lambda$ and $\mu'=\PPPP_m^*\hat\lambda'$. We see that the required contraction will be proved if we show that 
\begin{equation} \label{4.04}
\|\PPPP_1^*\mu-\PPPP_1^*\mu'\|_{\rm{var}}\le q<1. 
\end{equation}

{\it Step 2\/}. 
To prove~\eqref{4.04}, we first note that 
$$
\PPPP_1^*\mu-\PPPP_1^*\mu'
=\int_{X\times X}\bigl(P_1(u,\cdot)-P_1(u',\cdot)\bigr)\mu(\dd u)\mu'(\dd u'). 
$$
Taking the total variation norm and using~\eqref{4.2}, we derive  
\begin{align*}
\|\PPPP_1^*\mu-\PPPP_1^*\mu'\|_{\rm{var}}
&\le \int_{X\times X}\bigl\|P_1(u,\cdot)-P_1(u',\cdot)\bigr\|_{\rm{var}}\, 
\mu(\dd u)\mu'(\dd u')\\
&\le (\mu\otimes\mu')(G_\delta^c)+(1-\e)(\mu\otimes\mu')(G_\delta)\\
&=1-\e(\mu\otimes\mu')(G_\delta),
\end{align*}
where we set $G_\delta=B_X(\hat u,\delta)\times B_X(\hat u,\delta)$ and $G^c=(X\times X)\setminus G$. It remains to note that, by~\eqref{4.1}, we have 
$$
(\mu\otimes\mu')(G_\delta)\ge \mu(B_X(\hat u,\delta)) \mu'(B_X(\hat u,\delta))\ge p^2, 
$$
and therefore~\eqref{4.04} holds with $q=1-\e p^2$ for any $\lambda,\lambda'\in\PP(X)$. This completes the proof of Theorem~\ref{t4.1}. 
\end{proof}

\subsection{Image of measures under regular mappings}
\label{s4.2}
Let~$E$ be a separable Banach space, let~$X$ be a compact metric space, and let~$Y$ be a Riemannian manifold. We consider a continuous mapping $f:X\times E\to Y$ and recall that the concept of a decomposable measure is defined in Section~\ref{s1.1}. The following proposition is a particular case of  more general results established in Chapter~9 of~\cite{bogachev2010}.

\begin{proposition} \label{p4.3}
Let us assume that the mapping $f(u,\cdot):E\to Y$ is Fr\'echet differentiable for any fixed $u\in X$, the derivative $(D_\eta f)(u,\eta)$ is continuous on~$X\times E$, the image of the linear operator $(D_\eta f)(u_0,\eta_0)$ has full rank for some $(u_0,\eta_0)\in X\times E$, and~$\ell$ is a decomposable measure on~$E$ such that ${\mathsf P}_{n*}\ell$ possesses a positive continuous density with respect to the Lebesgue measure on~$F_n$. Then there is a ball $Q\subset X$ centred at~$u_0$ and a non-negative continuous function $\psi(u,y)$ defined on~$Q\times Y$ such that 
\begin{gather}
\psi(u_0,y_0)>0,\label{4.13}\\
f(u,\cdot)_*\ell\ge \psi(u,y)\vol(\dd y)\quad\mbox{for $u\in Q$},\label{4.14}
\end{gather}
where $y_0=f(u_0,\eta_0)$, and $\vol(\cdot)$ is the Riemannian measure on~$Y$. 
\end{proposition}

A simple direct proof of Proposition~\ref{p4.3} can be found in~\cite{shirikyan-jfa2007} in the case when~$Y$ is a finite-dimensional vector space (see Theorem~2.4). Extension to the case of a Riemannian manifold is straightforward. 

\addcontentsline{toc}{section}{References}
\def\cprime{$'$} \def\cprime{$'$}
  \def\polhk#1{\setbox0=\hbox{#1}{\ooalign{\hidewidth
  \lower1.5ex\hbox{`}\hidewidth\crcr\unhbox0}}}
  \def\polhk#1{\setbox0=\hbox{#1}{\ooalign{\hidewidth
  \lower1.5ex\hbox{`}\hidewidth\crcr\unhbox0}}}
  \def\polhk#1{\setbox0=\hbox{#1}{\ooalign{\hidewidth
  \lower1.5ex\hbox{`}\hidewidth\crcr\unhbox0}}} \def\cprime{$'$}
  \def\polhk#1{\setbox0=\hbox{#1}{\ooalign{\hidewidth
  \lower1.5ex\hbox{`}\hidewidth\crcr\unhbox0}}} \def\cprime{$'$}
  \def\cprime{$'$} \def\cprime{$'$} \def\cprime{$'$}
\providecommand{\bysame}{\leavevmode\hbox to3em{\hrulefill}\thinspace}
\providecommand{\MR}{\relax\ifhmode\unskip\space\fi MR }
\providecommand{\MRhref}[2]{%
  \href{http://www.ams.org/mathscinet-getitem?mr=#1}{#2}
}
\providecommand{\href}[2]{#2}

\end{document}